\documentclass[oneside,reqno]{amsart}

\usepackage{amsmath,amssymb,amsthm}
\usepackage[width=0.7\paperwidth]{geometry}
\usepackage[
    colorlinks=true,
    linkcolor=red,
    citecolor=green,
    urlcolor=blue
]{hyperref}

\numberwithin{equation}{section}

\newtheorem{theorem}{Theorem}[section]
\newtheorem{proposition}[theorem]{Proposition}
\newtheorem{lemma}[theorem]{Lemma}

\theoremstyle{definition}
\newtheorem{definition}[theorem]{Definition}
\newtheorem{remark}[theorem]{Remark}

\newcommand{\R}{\mathbb R}
\newcommand{\D}{\mathbb D}
\newcommand{\T}{\mathbb T}

\newcommand{\Q}{\mathbb Q}

\newcommand{\E}{\mathbb E}
\newcommand{\Ent}{\mathcal H}
\newcommand{\Ren}{\mathcal D}
\newcommand{\FI}{\mathcal I}

\newcommand{\diag}{\Delta}
\newcommand{\dd}{\mathrm d}

\newcommand{\WN}{\mathcal W_N}
\newcommand{\EN}{\mathcal E_N}
\newcommand{\CN}{\mathcal C_N}
\newcommand{\QN}{\overline{\Q}_{N,\beta}}
\newcommand{\ZN}{\overline{Z}_{N,\beta}}

\begin{document}

\author[X Feng]{Xuanrui Feng}
\address{School of Mathematical Sciences, Peking University, Beijing 100871, China.}
\email{pkufengxuanrui@stu.pku.edu.cn}

\author[Z Wang]{Zhenfu Wang}
\address{Beijing International Center for Mathematical Research, Peking University, Beijing 100871, China}
\email{zwang@bicmr.pku.edu.cn}

\title[Mean-Field Limit for 3D Vlasov--Poisson--Fokker--Planck]
{Modulated Gibbs Measure and Mean-Field Limit for 3D Vlasov--Poisson--Fokker--Planck Equations}
\subjclass[2020]{35Q83, 35Q84, 82C22, 82C31}
\keywords{Vlasov--Poisson--Fokker--Planck equations, propagation of chaos, mean-field limit, modulated Gibbs measure}
\date{\today}

\begin{abstract}
    We introduce a modulated Gibbs measure for the usual tensorized initial data for stochastic Newton's systems with singular repulsive interactions. Using the uniform-in-$N$ partition-function estimates of Wang--Zhao \cite{wang2026uniform}, we show that the modulated Gibbs measure and its tensorized reference are $O(N^{-1})$-close in normalized relative entropy. A stability argument then transfers propagation of chaos from modulated Gibbs measures to tensorized data and other initial laws at the same entropy scale. Combined with the weighted BBGKY estimates of Bresch--Jabin--Soler \cite{bresch2025new}, this yields the first cutoff-free mean-field limit and propagation of chaos for the three-dimensional Vlasov--Poisson--Fokker--Planck (VPFP) equation on an $N$-independent short-time interval. 
\end{abstract}

\maketitle
\tableofcontents

\section{Introduction}

\subsection{Classical particle systems and mean-field equations}

We study the mean-field limit of $N$ particles in the phase space $\Omega=\D\times\R^d$, where $\D=\T^d$ or $\R^d$. Their positions $X_i$ and velocities $V_i$ evolve according to the  stochastic Newton's dynamics
\begin{equation}\label{eq:particle-SDE}
    \begin{cases}
        \dd X_i=V_i \dd t,\\
        \displaystyle \dd V_i= \bigg( \frac{1}{N} \sum_{j \neq i} K(X_i-X_j) \bigg) \dd t+ \sqrt{\frac{2}{\beta}} \dd B_t^i,  \quad i = 1, 2, \cdots, N. 
    \end{cases}
\end{equation}
Here $(B_{\cdot}^i)$ are $N$ independent standard $d$-dimensional Brownian motions, and $\beta>0$ is the inverse temperature. The interaction force has the form $K=-\nabla g$, where $g$ is a radial repulsive potential. The admissible class of kernels is specified below.

Let $z_i=(x_i,v_i)$ and $Z_N=(z_1, z_2, \cdots, z_N)$. By It\^o's formula, the $N$-particle density $f_N=f_N(t,Z_N)$ satisfies the forward Kolmogorov equation
\begin{equation}\label{eq:forward-kolmogorov}
    \partial_t f_N +\sum_{i=1}^N \bigg( v_i \cdot \nabla_{x_i} f_N +\frac{1}{N} \sum_{j \neq i} K(x_i-x_j) \cdot \nabla_{v_i} f_N \bigg) = \frac{1}{\beta} \sum_{i=1}^N \Delta_{v_i} f_N.
\end{equation}
We assume that the particles are indistinguishable, and thus the joint law $f_N$ is a symmetric probability measure on $\Omega^N$. Under some appropriate assumptions on the initial value $f_N^0$, as $N \to \infty$, the particle system \eqref{eq:particle-SDE} is expected to converge to its mean-field limit, which is well known as the Vlasov--Fokker--Planck equation
\begin{equation}\label{eq:vlasov}
    \partial_t f+v \cdot \nabla_x f +(K*\rho_f) \cdot \nabla_v f=\frac{1}{\beta} \Delta_v f,\quad \rho_f(x)=\int_{\R^d} f(x,v) \dd v.
\end{equation}
Denote by $f_0$ the initial data for \eqref{eq:vlasov}. In the classical mean-field theory, this convergence is described by the notion of propagation of chaos introduced by Kac \cite{kac1956foundations}. Denote by $f_{N,k}$ the $k$-particle marginal of the $N$-particle distribution $f_N$,  for any $1 \leq k \leq N$. We say that the sequence $(f_N)$ is $f$-chaotic if, for any fixed $k$, the following weak convergence holds:
\begin{equation}\label{eq:marginal-convergence}
    f_{N,k} \rightharpoonup f^{\otimes k} \text{ weakly in } \mathcal{P}(\Omega^k)
\end{equation}
as $N \to \infty$. Therefore, propagation of chaos is defined as follows: if initially $f_N^0$ is $f_0$-chaotic, then there exists $T>0$ such that for any time $t \in [0, T]$, the solution $f_N^t$ of \eqref{eq:forward-kolmogorov} is also $f_t$-chaotic, where $f_t$ is the solution of \eqref{eq:vlasov}. The sequence $(f_N)$ is $f$-chaotic if and only if the empirical measure
\begin{equation*}
    L_N=\frac{1}{N} \sum_{i=1}^N \delta_{(x_i,v_i)}
\end{equation*}
converges in law under $f_N$ to the deterministic measure $f$, i.e. 
\begin{equation}\label{eq:empirical-convergence}
    L_N \Rightarrow \delta_f \text{ in law in } \mathcal{P}(\Omega)
\end{equation}
as $N \to \infty$; see Sznitman \cite{sznitman1991topics}. Propagation of chaos for deterministic or stochastic classical particle systems, in particular when the kernel $K$ is singular, still remains a major challenge in mathematical kinetic theory. For the classical theory on mean-field limits and propagation of chaos, we refer to reviews such as \cite{hauray2014kac,golse2016dynamics,jabin2014review,jabin2017mean}. 

\subsection{Main results}

In the study of mean-field limits and propagation of chaos for classical particle systems such as \eqref{eq:particle-SDE}, one usually assumes $f_0$-chaotic initial data, or even fully tensorized data $f_N^0=f_0^{\otimes N}$, where $f_0$ is the initial datum for the mean-field equation, here the Vlasov--Fokker--Planck equation \eqref{eq:vlasov}. The main message of this article is that, for a large class of repulsive interaction kernels, it suffices to prove propagation of chaos for a suitably chosen initial law of the particle system, namely the modulated Gibbs measure associated with $f_0$.

Before presenting the exact definition of the modulated Gibbs measure, we clarify the range of admissible kernels considered throughout this article. We assume that
\begin{equation*}
    K=-\nabla g,
\end{equation*}
where the potential functions $g$ are of repulsive Riesz-type in the following form:
\begin{equation}\label{eq:admissible-potential}
    g_{d,s}=
    \begin{cases}
        -c_d \log |x|, \quad &s=0;\\
        c_{d,s} |x|^{-s}, &0<s<d/2
    \end{cases}
\end{equation}
for $\D=\R^d$, where $c_d, c_{d, s}$ are positive normalization constants. For $\D=\T^d$, the potentials shall be modified under a natural periodization and  we still use the same notations. This class includes the logarithmic potential in any dimension $d$ and the Coulomb potential when $d=3$.

Now we introduce the notion of the modulated Gibbs measure in the kinetic setting. For a spatial configuration $X^N = (x_1, \cdots, x_N)$, we define its empirical measure
\begin{equation*}
    \mu_N=\frac{1}{N} \sum_{i=1}^N \delta_{x_i}.
\end{equation*}
The modulated (potential) energy against the background probability measure $\rho \in \mathcal{P}(\mathbb{D})$, used by Duerinckx \cite{duerinckx2016mean} and Serfaty \cite{serfaty2020mean} in singular mean-field limits, is defined by
\begin{equation}\label{eq:modulated-energy}
    \WN(X_N;\rho)=\frac{1}{2} \iint_{\D^2 \setminus \diag} g(x-y) \dd (\mu_N-\rho)(x) \dd (\mu_N-\rho)(y),
\end{equation}
where $\diag$ is the diagonal of $\D^2$. The convergence $\WN(X_N;\rho)\to0$ provides a quantitative criterion for the weak convergence $\mu_N\rightharpoonup\rho$; see \cite[Proposition~3.6]{serfaty2020mean} for the Coulomb and super-Coulombic cases. We now define the modulated Gibbs measure as follows.

\begin{definition}[Modulated Gibbs measure]\label{def:modulated-gibbs}
    Given any bounded reference probability density $f$ on $\Omega$ (not necessarily a solution of \eqref{eq:vlasov}), let $\rho_f(x)=\int_{\R^d} f(x,v)\dd v$ denote its spatial marginal. The modulated Gibbs measure associated with $f^{\otimes N}$ with parameter $\beta$ is defined by
    \begin{equation*}
        \QN(f)= \frac{1}{\ZN} \exp \left( -\beta N \WN(X_N;\rho_f) \right) f^{\otimes N}(Z_N),
    \end{equation*}
    where $\WN(X_N;\rho_f)$ is the modulated energy. Here $\ZN$ is the partition function
    \begin{equation*}
        \ZN=\int_{\Omega^N} \exp \left( -\beta N \WN(X_N;\rho_f) \right) f^{\otimes N}(Z_N) \dd Z_N.
    \end{equation*}
\end{definition}

We prove that the modulated Gibbs measure is well defined and establish several useful properties in Section \ref{sec:modulated-gibbs-measure}. Now we state our first main result. Hereinafter $C_\beta$ denotes some positive constant that only depends on $\beta$ and may vary from line to line.

\begin{theorem}\label{thm:one-initial-data}
    Let $g=g_{d,s}$ in \eqref{eq:admissible-potential}. Let $f_0 \in L^1 \cap L^\infty(\Omega)$ be any bounded probability density function on $\Omega$ with its spatial density $\rho_{f_0} \in L^\infty(\D)$. If $s=0$, we further assume that there exist $a,b>0$ such that
    \begin{equation}\label{eq:exponential-integrability}
        M_{a,b}= \int_\Omega e^{a |x|^b} f_0(z) \dd z<\infty.
    \end{equation}
    Let $f_t$ be the unique smooth solution to \eqref{eq:vlasov} with initial value $f_0$. If the mean-field convergence \eqref{eq:marginal-convergence} holds for any entropy solution $g_N$ of \eqref{eq:forward-kolmogorov} with initial value $g_N^0=\QN(f_0)$ given by the modulated Gibbs measure associated to $f_0$ for any $t \in [0,T]$, then it also holds for any entropy solution $f_N$ of \eqref{eq:forward-kolmogorov} with any initial value $f_N^0$ that is $f_0$-chaotic in the sense of relative entropy (see Definition \ref{def:normalized-relative-entropy}):
    \begin{equation}\label{eq:initial-entropy-bound}
        \Ent_N (f_N^0|f_0^{\otimes N}) \leq \frac{C_\beta}{N}
    \end{equation}
    for any $t \in [0,T]$. In particular, it applies to $f_N$ with fully tensorized initial value $f_N^0=f_0^{\otimes N}$.
\end{theorem}

We refer to \cite[Subsection 2.4]{bresch2025new} for the explicit notion of entropy solutions. Roughly speaking, an entropy solution is a weak solution that dissipates the entropy. 

Theorem \ref{thm:one-initial-data} provides a general guiding principle for proving propagation of chaos and (singular) mean-field limits by identifying the modulated Gibbs measure $\QN(f_0)$ as a canonical initial law.  As shown in Section \ref{sec:modulated-gibbs-measure}, 
it satisfies the relative entropy bound \eqref{eq:initial-entropy-bound} and hence the usual $f_0$-chaos condition, making it a natural initial law to include in propagation of chaos results. Conversely, proving
propagation of chaos for every entropy solution starting from $\QN(f_0)$ suffices to establish it for all initial laws satisfying the same bound, including the tensorized law $f_0^{\otimes N}$.

The modulated Gibbs measure corresponds to a non-i.i.d. sampling of the initial configuration. When the potential $g$ is singular, near-collision configurations in this case are less likely than under tensorized sampling, since the density function is tilted by a modulated Gibbs weight. This observation motivates the present work, where the suppression of near-collisions by the Gibbs weight can facilitate the proof of propagation of chaos. In particular, combining Theorem \ref{thm:one-initial-data} with the estimates in Bresch--Jabin--Soler \cite{bresch2025new}, we establish the complete mean-field limit and propagation of chaos for the 3D Vlasov--Poisson--Fokker--Planck (VPFP) equation as follows.  

\begin{theorem}[Mean-field limit for 3D VPFP]\label{thm:mfl-vpfp}
    Let $g=g_{d,s}$ in \eqref{eq:admissible-potential} with $0<s<d/2$. Assume that $f_0 \in L^1 \cap L^\infty(\Omega)$ with $\rho_{f_0} \in L^\infty(\D)$ and that there exists $\alpha>0$ such that
    \begin{equation*}
        M_\alpha=
        \begin{cases}
            \displaystyle \int_\Omega e^{\alpha |v|^2} f_0(z) \dd z, \quad &\, \mbox{when} \, \, \D=\T^d\\
            \displaystyle \int_\Omega e^{\alpha |z|^2} f_0(z) \dd z, \quad &\,  \mbox{when} \,\,  \D=\R^d
        \end{cases}
    \end{equation*}
    is finite. Assume in addition that $f_0 \in C^\infty(\Omega)$ and that the Vlasov--Fokker--Planck equation \eqref{eq:vlasov} admits a unique smooth solution $f$ on $[0,T]$ with initial datum $f_0$ for some $T>0$. Then for any entropy solution $f_N$ of \eqref{eq:forward-kolmogorov} with any initial value $f_N^0$ satisfying
    \begin{equation*}
        \Ent_N (f_N^0|f_0^{\otimes N}) \leq \frac{M_\beta}{N}
    \end{equation*}
    for some constant $M_\beta>0$, there exists some $T^\ast \in (0,T]$ that depends only on $\beta, \|f_0\|_{L^\infty}, \|\rho_{f_0}\|_{L^\infty}$ and $M_\alpha, M_\beta$ such that 
    \begin{equation*}
        f_{N,k} \rightharpoonup f^{\otimes k} \text{ weakly in } \mathcal{P}(\Omega^k)
    \end{equation*}
    holds for any $t \in [0,T^\ast]$. In particular, taking $d=3$ and $s=d-2=1$, we have the mean-field limit for the 3D Vlasov--Poisson--Fokker--Planck equation on $[0,T^\ast]$.
\end{theorem}

To the best of our knowledge, Theorem \ref{thm:mfl-vpfp} gives the first propagation of chaos result for the 3D Vlasov--Poisson--Fokker--Planck equation on a fixed time interval. 

\begin{remark}[First-order systems]
    The reduction to modulated Gibbs initial data in Theorem \ref{thm:one-initial-data} also applies to the first-order system
    \begin{equation*}
        \dd X_i=\frac{1}{N} \sum_{j \neq i} K(X_i-X_j) \dd t+\sqrt{\frac{2}{\beta}} \dd B_t^i,
        \quad i=1, \cdots, N.
    \end{equation*}
    For a reference probability density $\rho$ on $\D$, the corresponding modulated Gibbs measure takes the form
    \begin{equation*}
        \QN(\rho)=\frac{1}{\ZN} \exp \left( -\beta N\WN(X_N;\rho) \right) \rho^{\otimes N}(X_N),
    \end{equation*}
    where $\WN(X_N;\rho)$ is given by \eqref{eq:modulated-energy} and $\ZN$ is the normalization constant. This measure has appeared in the first-order setting in \cite{bresch2019modulated,bresch2019mean,rosenzweig2025modulated}. The estimates in Section \ref{sec:modulated-gibbs-measure} and the proof of Theorem \ref{thm:one-initial-data} in Section \ref{sec:proof-main-results} extend to this setting with the corresponding modifications. We focus on the application in Theorem \ref{thm:mfl-vpfp} and omit the details of the first-order extension.
\end{remark}

\begin{remark}[Log and Bessel--Riesz potentials]\label{rmk:log-bessel-riesz}
    Theorem \ref{thm:mfl-vpfp} also applies to log potentials on the torus under the extra assumption of \eqref{eq:exponential-integrability}, and to a wider class of Bessel--Riesz kernels defined by
    \begin{equation*}
        \widehat{g_{d,s}^{(m)}}(\xi)=(m^2+|\xi|^2)^{-(d-s)/2}
    \end{equation*}
    for $\D=\R^d$, for any $0<s<d/2$ and any $m \geq 0$ without extra assumptions. Here $m=0$ corresponds to \eqref{eq:admissible-potential}. For $\D=\T^d$ the potentials shall be modified under a natural periodization and will still be denoted by the same notations. The partition function estimate holds as in Proposition \ref{prop:partition-function}, and all the other arguments follow identically. We take the Riesz potentials as typical examples in this article.
\end{remark}

We briefly sketch the proof strategy of our main results. First key oberservation that leads to Theorem \ref{thm:one-initial-data} is a stability-type estimate of entropy solutions of \eqref{eq:forward-kolmogorov}. In general one may not expect stability of entropy solutions since they are not uniquely determined by the initial data. Nevertheless, given two nearby initial data and any entropy solution starting from one, we can construct an entropy solution starting from the other that remains close to the prescribed solution. This is our Proposition \ref{prop:stability-type-estimates}. Then we transfer Kac's chaos from one solution to the other via a relative-entropy inequality, see Lemma \ref{lem:weak-convergence}. The other key technical estimate follows from the nice properties of the modulated Gibbs measure, which are gathered in Section \ref{sec:modulated-gibbs-measure} and essentially guaranteed by a uniform-in-$N$ partition function estimate for any $\beta>0$.

With Theorem \ref{thm:one-initial-data} at hand, it suffices to prove Theorem \ref{thm:mfl-vpfp} for the entropy solution $g_N$. Here we use the estimates of Bresch--Jabin--Soler \cite{bresch2025new}, where a compactness-uniqueness method is applied to prove the mean-field limit for the particle system \eqref{eq:particle-SDE}. The central idea is to obtain compactness of entropy solutions to \eqref{eq:forward-kolmogorov} under an exponentially weighted $L^q$ norm. They resolve the 2D Vlasov--Poisson--Fokker--Planck model together with a partial result for the 3D case. More explicitly, the main estimates of the BBGKY hierarchy hold in any dimension, but the weighted norm is not in general finite in 3D, even for tensorized initial data $f_0^{\otimes N}$ with bounded $f_0$. But for modulated Gibbs initial data $g_N^0=\QN(f_0)$, the negative singular term in $\WN(X_N;\rho_f)$ dominates the positive singular term in the exponent of the weight, and thus all estimates from \cite{bresch2025new} now work for $g_N$. 
 
\subsection{Related works}

We briefly review related works on the mean-field derivation of Vlasov-type equations from Newtonian particle systems such as \eqref{eq:particle-SDE}, in both stochastic and deterministic settings.

For Lipschitz interactions, classical coupling and stability arguments yield mean-field limits and propagation of chaos; see for instance McKean \cite{mckean1967propagation}, Dobrushin \cite{dobrushin1979vlasov} and Sznitman \cite{sznitman1991topics}. For deterministic and stochastic Vlasov systems with bounded forces, Jabin--Wang \cite{jabin2016mean} proved propagation of chaos by a new relative entropy method, assuming sufficient regularity of the limiting solutions. Carrillo--Choi--Hauray--Salem \cite{carrillo2019mean} proved quantitative mean-field convergence for deterministic Cucker--Smale models with bounded, discontinuous communication weights arising from sharp sensitivity regions, using weak-strong stability in the $1$-Wasserstein distance. For diffusive systems, Lacker \cite{lacker2023hierarchies} obtained optimal local relative entropy rates through an entropy hierarchy under transport-entropy assumptions, and this framework extends to kinetic systems.

For singular interactions with velocity diffusion, Hauray--Salem \cite{hauray2019propagation} proved quantitative propagation of chaos for the one-dimensional Coulomb force, both repulsive and attractive. Bresch--Jabin--Soler \cite[Theorem~2 and Proposition~5]{bresch2025new} developed weighted BBGKY estimates for repulsive forces $K=-\nabla g \in L^p(\T^d)$ for any $p>1$. Under exponential integrability of $g$, their method yields short-time convergence and covers the two-dimensional VPFP equation. 

A complementary approach uses the dual hierarchy. Bresch--Duerinckx--Jabin \cite{bresch2024duality} established propagation of chaos for odd forces $K\in L^2_{\mathrm{loc}}$, with or without velocity diffusion and without assuming a potential structure. An additional Sobolev regularity on $K$ also yields convergence rates. Khoury--Jabin \cite{khoury2026quantitative} further obtained short-time quantitative estimates for marginals and higher-order correlations under additional regularity assumptions. Duerinckx--Jabin \cite{duerinckx2026derivation} extended the force class to $L^{2-\eta}_{\mathrm{loc}}$ for sufficiently small $\eta>0$, using hypoelliptic estimates with diffusion and velocity averaging without diffusion. Their result includes the unregularized two-dimensional Vlasov--Poisson equation assuming that the limiting solution has enough required regularity.

For deterministic singular systems, Hauray--Jabin \cite{hauray2007n,hauray2015particle} established mean-field limits for forces of order $|x|^{-\alpha}$ with $\alpha<1$, including quantitative propagation of chaos in $d \geq 3$. Lazarovici--Pickl \cite{lazarovici2017mean} treated the three-dimensional Coulomb force with an $N$-dependent cutoff. Duerinckx--Jabin \cite{duerinckx2026singular} derived the linearized Vlasov dynamics for fluctuations around Gibbs equilibrium at sufficiently high temperatures, including three-dimensional repulsive Coulomb interactions.

For the three-dimensional VPFP equation with an $N$-dependent cutoff, quantitative limits were obtained by Carrillo--Choi--Salem \cite{carrillo2019propagation} and Huang--Liu--Pickl \cite{huang2020mean}. Chen--Jung--Pickl--Wang \cite{chen2026mean} combined trajectory estimates with relative entropy to obtain strong convergence of the marginals. 

Our main result  Theorem \ref{thm:mfl-vpfp} establishes the mean-field limit for the unregularized repulsive Coulomb force on an $N$-independent short-time interval, for tensorized initial data and the larger class \eqref{eq:initial-entropy-bound}. The key step is to satisfy the weighted BBGKY bounds with modulated Gibbs initial data and then transfer convergence to this class through Theorem \ref{thm:one-initial-data}. The technique of exponentially tilting the initial data may extend to other singular mean-field limits. The mean-field derivation of the VPFP equation over long time scales and the cutoff-free mean-field derivation of the three-dimensional Vlasov--Poisson equation remain open.

\section{Modulated Gibbs measure}\label{sec:modulated-gibbs-measure}

In this section, we establish several properties of the modulated Gibbs measure.

\subsection{Uniform partition-function estimates}

This subsection is devoted to a key uniform-in-$N$ estimate for the partition function $\ZN$ in Definition \ref{def:modulated-gibbs}, which we recall
\begin{equation*}
    \begin{aligned}
        \ZN=&\, \int_{\Omega^N} \exp \left( -\beta N \WN(X_N;\rho_f) \right) f^{\otimes N} (Z_N) \dd Z_N\\
        =&\, \int_{\D^N} \exp \left( -\beta N \WN(X_N;\rho_f) \right) \rho_f^{\otimes N} (X_N) \dd X_N.
    \end{aligned}
\end{equation*} 

For any bounded reference measure $f$ and any admissible potential $g$, we have the following uniform-in-$N$ partition function estimate for any positive $\beta$.

\begin{proposition}[Uniform-in-$N$ partition-function estimate]\label{prop:partition-function}
    Let $\beta>0$ and let $g=g_{d,s}$ in \eqref{eq:admissible-potential}. Assume that $f \in L^1 \cap L^\infty(\Omega)$ with  $\rho_f \in L^\infty(\D)$, together with the extra assumption \eqref{eq:exponential-integrability} if $s=0$. Then we have the following uniform-in-$N$ partition function estimate:
    \begin{equation*}
        \frac{1}{C_\beta}  \leq \ZN \leq C_\beta,
    \end{equation*}
    where $C_\beta$ depends only on $\beta$ and $\|\rho_f\|_{L^\infty(\D)}$, and also on $M_{a,b}$ if $s=0$.
\end{proposition}

\begin{proof}
    The lower bound follows from Jensen's inequality. For the upper bound, we refer to the recent work of Wang--Zhao \cite[Theorem 1.1]{wang2026uniform} applied to $\rho_f$.  When $\D=\R^d$ the result follows directly from the cited work, and for $\D=\T^d$ it also holds under a natural periodization.
\end{proof}

For related works on uniform partition-function estimates, we also refer to Jabin--Wang \cite{jabin2018quantitative} for bounded potentials, Delgadino--Gvalani \cite{delgadino2025sharp} for repulsive log potentials, Duerinckx--Jabin \cite{duerinckx2026singular} for small $\beta$ and the very recent work by Delgadino--Gvalani--Rosenzweig \cite{delgadino2026sharp}. 

\subsection{Modulated free energy as relative entropy}

The idea of the modulated Gibbs measure in the first-order setting arises from the modulated free energy method for the mean-field limit. We briefly review the related discussions by adapting to the kinetic setting.

We start with the definition of normalized relative entropy and modulated free energy.

\begin{definition}[Normalized relative entropy]\label{def:normalized-relative-entropy}
    Given any two probability measures $\mu$ and $\nu$ on the product space $E^N$ of a Polish space $E$, the normalized relative entropy between $\mu$ and $\nu$ is defined as
    \begin{equation*}
        \Ent_N (\mu|\nu)=
        \begin{cases}
              \displaystyle \frac{1}{N} \int_{E^N} \dd \mu \log \frac{\dd \mu}{\dd \nu}, &\, \mbox{if } \, \, \mu \ll \nu,\\
              \displaystyle \infty, &\, \text{otherwise}.
        \end{cases}
    \end{equation*}
\end{definition}

Motivated by the notion of free energy in physics, Bresch--Jabin--Wang \cite{bresch2019modulated,bresch2019mean,bresch2023mean} introduced the modulated free energy between $f_N$ and $f^{\otimes N}$ as follows.

\begin{definition}[Modulated free energy]\label{def:modulated-free-energy} 
The modulated free energy of $f_N$ relative to $f^{\otimes N}$ is defined by
\begin{equation*}
    \EN \left( f_N, f^{\otimes N} \right)=\frac{1}{\beta} \Ent_N \left( f_N|f^{\otimes N} \right)+ \E_{f_N} \WN(X_N;\rho_f).
\end{equation*}
\end{definition}

As observed by Rosenzweig--Serfaty \cite{rosenzweig2025modulated}, the modulated free energy equals $1/\beta$ times the normalized relative entropy of $f_N$ with respect to the modulated Gibbs measure $\QN$, up to a partition-function correction. Under the assumptions of Proposition \ref{prop:partition-function}, this correction is $O(N^{-1})$. We summarize it as the following proposition. 

\begin{proposition}\label{prop:modulated-free-energy-as-entropy}
    We have
    \begin{equation*}
      \EN \left( f_N, f^{\otimes N} \right)=\frac{1}{\beta} \Ent_N \left( f_N|\QN \right)- \frac{1}{\beta N} \log \ZN.
    \end{equation*}
    Moreover, under the assumptions of Proposition \ref{prop:partition-function}, we have
    \begin{equation*}
        \frac{1}{\beta} \Ent_N \left( f_N|\QN \right)-\frac{C_\beta}{N} \leq \EN \left( f_N, f^{\otimes N} \right) \leq \frac{1}{\beta} \Ent_N \left( f_N|\QN \right)+\frac{C_\beta}{N}.
    \end{equation*}
\end{proposition}

The first inequality recovers the coercivity estimate in \cite{bresch2019mean} in the repulsive regime. See also pointwise asymptotic positivity estimates of the modulated energy in \cite{serfaty2020mean}.

We proceed to compute the time derivative of the modulated free energy under \eqref{eq:forward-kolmogorov} and \eqref{eq:vlasov}. We start with the definition of normalized relative Fisher information in the velocity variables

\begin{definition}[Normalized relative Fisher information in the velocity variables]\label{def:normalized-relative-fisher}
    Given any two probability measures $\mu$ and $\nu$ on the product space $\Omega^N$, the normalized relative Fisher information in the velocity variables between $\mu$ and $\nu$ is defined as
    \begin{equation*}
        \FI_N^v (\mu|\nu)=
        \begin{cases}
            \displaystyle \frac{1}{N} \sum_{i=1}^N \int_{\Omega^N} \dd \mu \left| \nabla_{v_i} \log \frac{\dd \mu}{\dd \nu} \right|^2, &\,  \mbox{if\, }\, \mu \ll \nu,\\
            \displaystyle \infty, &\, \text{otherwise}.
        \end{cases}
    \end{equation*}
\end{definition}

The following formal calculation illustrates the dissipation structure of the modulated free energy. The appearance of the relative Fisher information $\FI_N^v (f_N|\QN)$ is suggested by Proposition \ref{prop:modulated-free-energy-as-entropy}.

\begin{proposition}[Evolution of modulated free energy]\label{prop:evolution-modulated-free-energy}
    Let $f_N$ and $f$ be classical solutions of the forward Kolmogorov equation \eqref{eq:forward-kolmogorov} and the Vlasov--Fokker--Planck equation \eqref{eq:vlasov}, respectively. Then the modulated free energy satisfies the evolution identity
    \begin{equation*}
        \frac{\dd}{\dd t}\EN \left( f_N, f^{\otimes N} \right)=-\frac{1}{\beta^2} \FI_N^v (f_N|\QN)+\E_{f_N} \CN[a_f].
    \end{equation*}
    The commutator-type functional $\CN$ is given by
    \begin{equation*}
        \CN[a]= \frac{1}{2} \iint_{\Omega^2 \setminus \Delta} \nabla g(x-y) \cdot \left( a(z)-a(z') \right) \dd (L_N-f)^{\otimes 2} (z,z'),
    \end{equation*}
    where $L_N$ is the empirical measure and the function $a_f$ is given by
    \begin{equation*}
        a_f(z)=v+ \frac{1}{\beta} \nabla_v \log f(z).
    \end{equation*}
\end{proposition}

\begin{proof}
    We compute the time derivative of the relative entropy part and the modulated energy part separately. A direct relative entropy computation gives
    \begin{equation*}
        \begin{aligned}
            \frac{\dd}{\dd t} \Ent_N \left( f_N|f^{\otimes N} \right) = &\, -\frac{1}{\beta} \FI_N^v \left( f_N|f^{\otimes N} \right)\\
            &\, -\frac{1}{2} \E_{f_N} \iint_{\Omega^2 \setminus \Delta} K(x-y)\cdot \left( \nabla_v \log f(z)-\nabla_w \log f(z') \right) \dd (L_N-f)^{\otimes 2} (z,z').
        \end{aligned}
    \end{equation*}
    The modulated energy computation shows
    \begin{equation*}
        \frac{\dd}{\dd t} \WN(X_N;\rho_f)= -\frac{1}{2} \iint_{\Omega^2 \setminus \Delta} K(x-y) \cdot (v-w) \dd (L_N-f)^{\otimes 2} (z,z').
    \end{equation*}
    Since the Gibbs tilt depends only on $x$-variables, we simply have
    \begin{equation*}
        \nabla_{v_i} \log \frac{f_N}{\QN}=\nabla_{v_i} \log \frac{f_N}{f^{\otimes N}},
    \end{equation*}
    and thus $\displaystyle \FI_N^v \left( f_N|f^{\otimes N} \right)=\FI_N^v (f_N|\QN)$. This completes the proof.
\end{proof}

A similar observation was also made by Rosenzweig--Serfaty \cite{rosenzweig2025modulated} for first-order systems, where they aimed to prove a uniform log-Sobolev inequality for the modulated Gibbs measure. 

\subsection{Relative entropy bound}

In this subsection, we show that the modulated Gibbs measure $\QN$ is actually $f$-chaotic in a strong sense. It is direct to estimate in the sense of relative entropy.

\begin{lemma}[Relative entropy estimates]\label{lem:relative-entropy}
    Let $\beta>0$ and let $g=g_{d,s}$ in \eqref{eq:admissible-potential}. Assume that $f \in L^1 \cap L^\infty(\Omega)$ with $\rho_f \in L^\infty(\D)$. Assume further that \eqref{eq:exponential-integrability} holds  if $s=0$. Then we have
    \begin{equation*}
        \Ent_N \left( \QN|f^{\otimes N} \right) \leq \frac{C_\beta}{N}, \quad \Ent_N \left( f^{\otimes N}|\QN \right) \leq \frac{C_\beta}{N}.
    \end{equation*}
\end{lemma}

\begin{proof}
    A direct computation shows
    \begin{align*}
        \Ent_N \left( f^{\otimes N}|\QN \right)=&\, \frac{1}{N} \int_{\Omega^N} f^{\otimes N} \log \left( \ZN \exp \left( \beta N \WN(X_N;\rho_f) \right) \right)\\
        =&\, \frac{\log \ZN}{N}+\beta \E_{f^{\otimes N}} \WN(X_N;\rho_f).
    \end{align*}
    The first term is bounded by $C_\beta/N$ by Proposition \ref{prop:partition-function}, and the second term is bounded by $C_\beta/N$ using the simple computation
    \begin{equation*}
        \E_{f^{\otimes N}} \WN(X_N;\rho_f)=-\frac{1}{2N} \int_\D g \ast \rho_f \rho_f,
    \end{equation*}
    since $\WN(X_N;\rho_f)$ is almost centered with respect to $\rho_f$ and $\rho_f \in L^1 \cap L^\infty(\D)$ by assumption. Similarly, we have
    \begin{align*}
        \Ent_N \left( \QN|f^{\otimes N} \right)=&\, \frac{1}{N} \int_{\Omega^N} \QN \log \left( \frac{1}{\ZN} \exp \left( -\beta N \WN(X_N;\rho_f) \right) \right)\\
        =&\, -\frac{\log \ZN}{N}-\beta \E_{\QN} \WN(X_N;\rho_f).
    \end{align*}
    The first term is bounded by $C_\beta/N$ by Proposition \ref{prop:partition-function} again, and the second term is estimated by applying the Donsker--Varadhan variational inequality (see \cite[Lemma 1]{jabin2018quantitative})
    \begin{equation*}
        -\E_{\QN} \WN(X_N;\rho_f) \leq \frac{1}{2\beta} \Ent_N \left( \QN|f^{\otimes N} \right)+\frac{1}{2\beta N} \log \E_{f^{\otimes N}} \exp \left( -2\beta N \WN(X_N;\rho_f) \right).
    \end{equation*}
    By Proposition \ref{prop:partition-function}, the last term is bounded by $C_\beta/N$, and the result follows.
\end{proof}

We also introduce the notion of R\'enyi divergence.

\begin{definition}[R\'enyi divergence]\label{def:renyi-divergence}
    Given any $p>1$ and any two probability measures $\mu$ and $\nu$ on $E^N$, the $p$-R\'enyi divergence between $\mu$ and $\nu$ is defined as
    \begin{equation*}
        \Ren_p (\mu|\nu)= \frac{1}{p-1} \log \int_{E^N} \dd \nu \left( \frac{\dd \mu}{\dd \nu}\right)^p.
    \end{equation*}
\end{definition}

A direct computation and the uniform partition-function estimates show that $\QN$ is $f$-chaotic in the sense of R\'enyi divergence.

\begin{lemma}[R\'enyi divergence estimates]\label{lem:renyi-divergence}
    Let $\beta>0$ and let $g=g_{d,s}$ in \eqref{eq:admissible-potential} or $g=g_{d,s}^{(m)}$. Assume that $f \in L^1 \cap L^\infty(\Omega)$ with $\rho_f \in L^\infty(\D)$, together with the extra assumption \eqref{eq:exponential-integrability} if $s=0$. Then we have
    \begin{equation*}
        \Ren_p \left( \QN|f^{\otimes N} \right)= \frac{1}{p-1} \left( \log Z_{N,p\beta}-p \log \ZN \right) \leq C_{\beta,p}.
    \end{equation*}
\end{lemma}

By the triangle-like inequality for R\'enyi divergence (see for instance \cite[Lemma 1.4.3]{dupuis1997weak}), we have
\begin{equation}\label{eq:triangle-formula}
    \Ent_N \left( f_N|f^{\otimes N} \right) \leq \frac{p}{p-1} \Ent_N (f_N|\QN)+\frac{1}{N} \Ren_p \left( \QN|f^{\otimes N} \right) \leq \frac{p}{p-1} \Ent_N (f_N|\QN)+\frac{C_{\beta,p}}{N}.
\end{equation}
This shows that a bound on $\Ent_N(f_N|\QN)$ in turn controls the classical normalized relative entropy. This coincides with the modulated free energy method in \cite{bresch2019modulated,bresch2019mean}, where the authors proved that the modulated free energy does control the relative entropy by exploiting pointwise lower bounds of modulated energy in \cite{serfaty2020mean}.

\section{Proof of main results}\label{sec:proof-main-results}

In this section, we prove our main results.

\subsection{Proof of Theorem \ref{thm:one-initial-data}}

The starting point is a stability-type estimate for the entropy solution of \eqref{eq:forward-kolmogorov}. 

\begin{proposition}[Stability-type estimate]\label{prop:stability-type-estimates}
    Consider any pair of  bounded symmetric probability density functions $f_N^0$ and $g_N^0$ on $\Omega^N$ such that
    \begin{equation*}
        \Ent_N (f_N^0|g_N^0) \leq \frac{C_\beta}{N}.
    \end{equation*}
    Let $f_N$ be any entropy solution of the forward Kolmogorov equation \eqref{eq:forward-kolmogorov} with initial data $f_N^0$. Then there exists an entropy solution of \eqref{eq:forward-kolmogorov} with initial data $g_N^0$, denoted by $g_N$, such that
    \begin{equation*}
        \Ent_N (f_N|g_N)(t) \leq \frac{C_\beta}{N}
    \end{equation*}
    for any $t \geq 0$.
\end{proposition}

\begin{remark}
    One may distinguish this result from the standard stability results. Proposition \ref{prop:stability-type-estimates} does not assert stability for arbitrary pairs of entropy solutions. For each entropy solution $f_N$ with initial datum $f_N^0$, one can choose an entropy solution $g_N$ with initial datum $g_N^0$ such that the relative entropy bound holds for all time. We do not require uniqueness of entropy solutions.
\end{remark}

\begin{proof}
    By the construction of entropy solutions of \eqref{eq:forward-kolmogorov}, there exists a family of regularized solutions $f_{N,\varepsilon}$ with initial data $f_N^0$ such that
    \begin{equation*}
        \partial_t f_{N,\varepsilon} +\sum_{i=1}^N \bigg( v_i \cdot \nabla_{x_i}+\frac{1}{N} \sum_{j \neq i} K_\varepsilon (x_i-x_j) \cdot \nabla_{v_i} \bigg) f_{N,\varepsilon}=\frac{1}{\beta} \sum_{i=1}^N \Delta_{v_i} f_{N,\varepsilon}.
    \end{equation*}
    Here $K_\varepsilon$ is a standard regularization of $K$. Then $f_N$ is obtained by taking a weak limit of a subsequence, still denoted by $f_{N,\varepsilon}$. Let $g_{N,\varepsilon}$ be the classical solution of the same regularized equation with initial data $g_N^0$. Then we have
    \begin{align*}
        \frac{\dd}{\dd t} \Ent_N (f_{N,\varepsilon}|g_{N,\varepsilon})=-\frac{1}{\beta} \FI_N^v (f_{N,\varepsilon}|g_{N,\varepsilon}) \leq 0
    \end{align*}
    by direct computations. Therefore, the desired bound holds for all pairs of regularized solutions uniformly in $\varepsilon$. Take $g_N$ as the weak limit of a subsequence of $g_{N,\varepsilon}$. By the joint lower semicontinuity of relative entropy, we immediately have
    \begin{equation*}
        \Ent_N (f_N|g_N)(t) \leq \liminf_{\varepsilon \to 0} \Ent_N (f_{N,\varepsilon}|g_{N,\varepsilon})(t) \leq \frac{C_\beta}{N}.
    \end{equation*}
    The result follows since $g_N$ is a desired entropy solution by definition.
\end{proof}

The next lemma shows that under the $O(1/N)$ bound in the sense of relative entropy in Proposition \ref{prop:stability-type-estimates}, if the sequence $g_N$ is $f$-chaotic, then $f_N$ is also $f$-chaotic. Therefore, the relative entropy bound allows us to transfer chaos from $g_N$ to $f_N$. 

\begin{lemma}[Transfer of chaos]\label{lem:weak-convergence}
	For each $N\geq 1$, let $f_N$ and $g_N$ be bounded symmetric probability densities on $\Omega^N$ satisfying
	\begin{equation*}
		\Ent_N(f_N|g_N)\leq \frac{C_\beta}{N},
	\end{equation*}
	where $C_\beta$ is independent of $N$. If the sequence $g_N$ is $f$-chaotic for some bounded probability density $f$ on $\Omega$, then $f_N$ is also $f$-chaotic.
\end{lemma}

\begin{proof}
    Recall that $\displaystyle L_N=\frac{1}{N} \sum_{i=1}^N \delta_{z_i}$ is the empirical measure. Let $\mathbf{d}(\cdot, \cdot)$ be any metric on $\mathcal P(\Omega)$ that metrizes weak convergence, for instance the bounded Lipschitz metric. For any $\varepsilon>0$, let
    \begin{equation*}
        A_{N,\varepsilon}=\{ \mathbf{d}(L_N,f) >\varepsilon\}.
    \end{equation*}
    By the classical result of Sznitman \cite{sznitman1991topics}, since $g_N$ is $f$-chaotic, we have $q_N=g_N(A_{N,\varepsilon}) \to 0$ as $N \to \infty$. Define $p_N=f_N(A_{N,\varepsilon})$. By the Donsker--Varadhan variational inequality, for any $\eta>0$, we have
    \begin{equation*}
        p_N \leq \frac{N}{\eta} \Ent_N (f_N|g_N) +\frac{1}{\eta} \log \int \exp (\eta \mathsf{1}_{A_{N,\varepsilon}}) g_N=\frac{N}{\eta} \Ent_N (f_N|g_N) +\frac{1}{\eta} \log \left( e^\eta q_N+(1-q_N) \right).
    \end{equation*}
    Choosing $\eta=\log(1+1/q_N)$, we have
    \begin{equation*}
        p_N \leq \frac{N \Ent_N(f_N|g_N)+\log 2}{\log(1+1/q_N)}.
    \end{equation*}
    See also Yau \cite[Equation (2.18)]{yau1991relative}. By assumption, the numerator is bounded in $N$ and the denominator diverges to $+\infty$ as $N \to \infty$. Therefore, we get $p_N \to 0$ as $N \to \infty$ for any $\varepsilon>0$, and this implies that $f_N$ is $f$-chaotic again by the classical result of Sznitman.
\end{proof}

Now we can finish the proof of our first main result.

\begin{proof}[Proof of Theorem \ref{thm:one-initial-data}]
    Suppose that the mean-field convergence \eqref{eq:marginal-convergence} holds for any entropy solution $g_N$ of \eqref{eq:forward-kolmogorov} with initial value $g_N^0=\QN(f_0)$ for any $t \in [0,T]$. 

    We first prove the result for the fully tensorized initial value $f_N^0=f_0^{\otimes N}$. Applying Lemma \ref{lem:relative-entropy} to $f_0$, we get
    \begin{equation*}
        \Ent_N (f_N^0|g_N^0)=\Ent_N (f_0^{\otimes N}|\QN) \leq \frac{C_\beta}{N}.
    \end{equation*}
    Therefore, for any entropy solution $f_N$ of \eqref{eq:forward-kolmogorov} with initial value $f_N^0=f_0^{\otimes N}$, by Proposition \ref{prop:stability-type-estimates}, there exists an entropy solution of \eqref{eq:forward-kolmogorov} with initial value $g_N^0$, denoted by $g_N$, such that
    \begin{equation*}
        \Ent_N (f_N|g_N) (t) \leq \frac{C_\beta}{N}
    \end{equation*}
    for any $t \geq 0$. By assumption, $g_N(t)$ is $f_t$-chaotic for any $t \in [0,T]$. Therefore, by Lemma \ref{lem:weak-convergence}, we conclude that $f_N(t)$ is $f_t$-chaotic for any $t \in [0,T]$. 

    Then we proceed to prove the result for any entropy solution $h_N$ of \eqref{eq:forward-kolmogorov} with any initial value $h_N^0$ satisfying the relative entropy bound \eqref{eq:initial-entropy-bound}. By Proposition \ref{prop:stability-type-estimates}, there exists an entropy solution of \eqref{eq:forward-kolmogorov} with initial value $f_0^{\otimes N}$, denoted by $f_N$, such that
    \begin{equation*}
        \Ent_N (h_N|f_N) (t) \leq \frac{C_\beta}{N}
    \end{equation*}
    for any $t \geq 0$. From the previous step, we know that $f_N(t)$ is $f_t$-chaotic for any $t \in [0,T]$. Therefore, by Lemma \ref{lem:weak-convergence}, we conclude that $h_N(t)$ is $f_t$-chaotic for any $t \in [0,T]$. This completes the proof.
\end{proof}

\subsection{Proof of Theorem \ref{thm:mfl-vpfp}}

Now we proceed to prove our second main result. By Theorem \ref{thm:one-initial-data}, we only need to prove that there exists $T^\ast>0$ such that $g_N(t)$ is $f_t$-chaotic for any $t \in [0,T^\ast]$ for any entropy solution $g_N$ of \eqref{eq:forward-kolmogorov} with initial value $g_N^0=\QN(f_0)$. 

For $\D=\T^d$ we borrow the following technical proposition from \cite{bresch2025new} adapted to our notation.

\begin{proposition}[\cite{bresch2025new}, Proposition 5]\label{prop:bjs-estimate}
    Let $K \in L^p(\T^d)$ for some $p>1$ and define
    \begin{equation*}
        \lambda(t)=\frac{1}{\Lambda(1+t)}, \quad L=\frac{C}{\lambda(1)^\theta} \|K\|_{L^p}^q
    \end{equation*}
    for some positive constants $\Lambda, C, \theta$ depending only on $q, d, \beta$ and provided that $1/q+1/p \leq 1$ with $q>2$. Consider any entropy solution $g_N$ of \eqref{eq:forward-kolmogorov} with initial value $g_N^0 \in L^\infty(\Omega^N)$ and satisfying
    \begin{equation}\label{eq:bjs-condition}
        \int_{\Omega^k} |g_{N,k}^0|^q e^{\lambda(0) e_k} \dd Z_k \leq F_0^k, \quad \sup_{t \leq 1} \int_{\Omega^N} |g_N(t)|^q e^{\lambda(t) e_N} \dd Z_N \leq F^N
    \end{equation}
    for some $F, F_0>0$, where
    \begin{equation*}
        e_k(Z_k)=\sum_{i=1}^k (1+|v_i|^2) +\frac{1}{N} \sum_{i,j=1}^k g(x_i-x_j), 
    \end{equation*}
where we use the convention that $g(0)=0$.  

   Then we have
    \begin{equation*}
        \sup_{t \leq T^\ast} \int_{\Omega^k} |g_{N,k}(t)|^q e^{\lambda(t) e_k} \dd Z_k \leq (2F_0)^k+(4F)^k 2^{-N-1}
    \end{equation*}
    for some $T^\ast>0$ depending only on $L, F_0, F$.
\end{proposition}

A similar result also holds for $\D=\R^d$ with minor modifications in the assumption for $K$ and the definition of $e_k$ and $\lambda(t)$ and $L$. We state the clear version here and present the proof, but most of the computations are exactly identical to those in \cite{bresch2025new}.

\begin{proposition}[\cite{bresch2025new}, Proposition 5, whole-space version]\label{prop:bjs-estimate-whole-space}
    Let $K=K_s+K_b$ with $K_s=K\mathsf{1}_{B_1} \in L^p(B_1)$ for some $p>1$ and $K_b=K\mathsf{1}_{B_1^c} \in L^\infty(B_1^c)$, where $B_1$ is the unit ball in $\R^d$. We define
    \begin{equation*}
        \lambda(t)=\frac{e^{-t}}{\Lambda(1+B(1-e^{-t}))}, \quad L=\frac{C}{\lambda(1)^\theta} (\|K_s\|_{L^p}^q+\|K_b\|_{L^\infty}^q)
    \end{equation*}
    for some positive constants $\Lambda, C, \theta$ depending only on $q, d, \beta$, and $B=c_0/\Lambda$ for some $c_0$ depending only on $q, \beta$, and provided that $1/q+1/p \leq 1$ with $q>2$. Consider any entropy solution $g_N$ of \eqref{eq:forward-kolmogorov} with initial value $g_N^0 \in L^\infty(\Omega^N)$ and satisfying
    \begin{equation}\label{eq:bjs-condition-whole-space}
        \int_{\Omega^k} |g_{N,k}^0|^q e^{\lambda(0) e_k} \dd Z_k \leq F_0^k, \quad \sup_{t \leq 1} \int_{\Omega^N} |g_N(t)|^q e^{\lambda(t) e_N} \dd Z_N \leq F^N
    \end{equation}
    for some $F, F_0>0$, where
    \begin{equation*}
        e_k(Z_k)=\sum_{i=1}^k (1+|x_i|^2+|v_i|^2) +\frac{1}{N} \sum_{i,j=1}^k g(x_i-x_j), 
    \end{equation*}
where again we use the convention that $g(0)=0$. 
    Then we have
    \begin{equation*}
        \sup_{t \leq T^\ast} \int_{\Omega^k} |g_{N,k}(t)|^q e^{\lambda(t) e_k} \dd Z_k \leq (2F_0)^k+(4F)^k 2^{-N-1}
    \end{equation*}
    for some $T^\ast>0$ depending only on $L, F_0, F$.
\end{proposition}

\begin{proof}[Proof of Proposition \ref{prop:bjs-estimate-whole-space}]
    We start with a computation similar to \cite[Lemma 9]{bresch2025new}. Define
    \begin{equation*}
        \mathcal{L}_k=\sum_{i=1}^k v_i \cdot \nabla_{x_i}+\frac{1}{N} \sum_{i,j=1}^k K(x_i-x_j) \cdot \nabla_{v_i}.
    \end{equation*}
    Applying the BBGKY hierarchy solved by $g_{N,k}$, we obtain
    \begin{align*}
        \frac{\dd}{\dd t} \int |g_{N,k}|^q e^{\lambda(t) e_k}=&\, \int |g_{N,k}|^q \left( \mathcal{L}_k e^{\lambda(t) e_k} \right)+\frac{q}{\beta} \int |g_{N,k}|^{q-1} \, \mathrm{sgn} \, g_{N,k} \bigg( \sum_{i=1}^k \Delta_{v_i} g_{N,k} \bigg) e^{\lambda(t) e_k}\\
        -&\, \frac{N-k}{N} q \sum_{i=1}^k \int |g_{N,k}|^{q-1} \, \mathrm{sgn} \, g_{N,k} \left( \nabla_{v_i} \cdot \int K(x_i-x_{k+1}) g_{N,k+1} \dd z_{k+1} \right) e^{\lambda(t) e_k}\\
        +&\, \lambda'(t) \int |g_{N,k}|^q e_k e^{\lambda(t) e_k}.
    \end{align*}
    Since $\displaystyle \mathcal{L}_k e^{\lambda(t) e_k}=2 \lambda(t) e^{\lambda(t) e_k} \sum_{i=1}^k v_i \cdot x_i$, the first term is bounded by
    \begin{equation*}
        \lambda(t) \int |g_{N,k}|^q e_k e^{\lambda(t) e_k}
    \end{equation*}
    by the Cauchy--Schwarz inequality. The remaining two terms are estimated by repeatedly using integration by parts and the Cauchy--Schwarz inequality in exactly the same way as in \cite[Lemma 9]{bresch2025new}. We arrive at
    \begin{align*}
        \frac{\dd}{\dd t} \int |g_{N,k}|^q e^{\lambda(t) e_k} &\, \leq \left[ \lambda'(t)+\lambda(t)+ \left( \frac{2q}{(q-1)\beta}+q \right) \lambda(t)^2 \right] \int |g_{N,k}|^q e_k e^{\lambda(t) e_k}\\
        +&\, \left( \frac{\beta}{2}q(q-1) +q \right) \frac{N-k}{N} \sum_{i=1}^k \int |g_{N,k}|^{q-2} \left| \int K(x_i-x_{k+1}) g_{N,k+1} \dd z_{k+1} \right|^2 e^{\lambda(t) e_k}.
    \end{align*}
    For simplicity, we denote by $c_1=\frac{2q}{(q-1)\beta}+q$ and $c_2=\frac{\beta}{2}q(q-1)+q$, which both depend only on $q$ and $\beta$. Applying H\"older's inequality to the integral in the second line, we bound the second line by
    \begin{equation*}
        c_2 k \frac{q-2}{q} \lambda(t)^2 \int |g_{N,k}|^q e^{\lambda(t) e_k}+\frac{2c_2}{q \lambda(t)^{q-2}} \frac{N-k}{N} \sum_{i=1}^k \int e^{\lambda(t) e_k} \left| \int K(x_i-x_{k+1}) g_{N,k+1} \dd z_{k+1} \right|^q.
    \end{equation*}
    The first term should be combined with the first line in the previous inequality, and the whole coefficient becomes
    \begin{equation*}
        \lambda'(t)+\lambda(t)+c_0\lambda(t)^2, \quad c_0=c_1+\frac{q-2}{q}c_2.
    \end{equation*}
    By definition of $\lambda(t)$, the coefficient in the first line is non-positive. For the remaining second term, we apply H\"older's inequality again to the integrand and use the assumptions on $K$ to get 
    \begin{align*}
        \left| \int K(x_i-x_{k+1}) g_{N,k+1} \dd z_{k+1} \right|^q \leq &\, \left( \int |K(x_i-x_{k+1})|^{q^\ast} e^{-\frac{q^\ast}{q} \lambda(t) |z_{k+1}|^2} \dd z_{k+1} \right)^{q/q^\ast}  \\
        &\, \qquad \qquad \cdot \int |g_{N,k+1}|^q e^{\lambda(t) |z_{k+1}|^2} \dd z_{k+1}.
    \end{align*}
    In the first factor, we first integrate over $v_{k+1}$, and then split the kernel into $K=K_s+K_b$ and bound the two parts by the $L^{q^\ast}$ norm and the $L^\infty$ norm. Thanks to the Gaussian weight in both $x_{k+1}$ and $v_{k+1}$, the first factor is bounded by
    \begin{equation*}
        C_{q,d} \lambda(t)^{-\theta_{q,d}} \left( \|K_s\|_{L^p}^q+\|K_b\|_{L^\infty}^q \right).
    \end{equation*}
    By simple bookkeeping, we conclude that
    \begin{equation*}
        \frac{\dd}{\dd t} \int |g_{N,k}|^q e^{\lambda(t) e_k} \leq C_{q,d,\beta} \lambda(t)^{-\theta_{q,d}} \frac{k(N-k)}{N} \left( \|K_s\|_{L^p}^q+\|K_b\|_{L^\infty}^q \right) \int |g_{N,k+1}|^q e^{\lambda(t) e_{k+1}}.
    \end{equation*}
    The remaining parts are identical to the proof of \cite[Proposition 5]{bresch2025new}.
\end{proof}

Now we finish the proof of our second main result.

\begin{proof}[Proof of Theorem \ref{thm:mfl-vpfp}]
    We apply the above two propositions to our setting. The integrability condition of $K$ is obviously satisfied. We only need to verify the conditions in \eqref{eq:bjs-condition} and \eqref{eq:bjs-condition-whole-space}. Recall that
    \begin{equation*}
        g_N^0=\QN(f_0)=\frac{1}{\ZN} \exp \left( -\beta N \WN(X_N;\rho_{f_0}) \right) f_0^{\otimes N}(Z_N).
    \end{equation*}
    By H\"older's inequality and Proposition \ref{prop:partition-function}, we have
    \begin{align*}
        \int_{\Omega^k} |g_{N,k}^0|^q e^{\lambda(0)e_k} \dd Z_k &\, \leq \frac{\|f_0\|_{L^\infty}^{(q-1)k}}{\ZN^q} \int_{\Omega^N} \exp \left( -q\beta N \WN(X_N;\rho_{f_0}) \right) f_0^{\otimes N} e^{\Lambda^{-1} e_k} \dd Z_N\\
        &\, \leq C_{\beta, q} \|f_0\|_{L^\infty}^{(q-1)k} \int_{\Omega^N} \exp \left( -q\beta N \WN(X_N;\rho_{f_0}) \right) f_0^{\otimes N} e^{\Lambda^{-1} e_k} \dd Z_N.
    \end{align*}
    Next we need to establish a lower bound for $N\WN(X_N;\rho_{f_0})$. Expand the integral and divide the variables into two groups, we write
    \begin{align*}
        N\WN(X_N;\rho_{f_0})=&\, \frac{1}{2N} \sum_{i \neq j}^N g(x_i-x_j)- \sum_{i=1}^N g \ast \rho_{f_0}(x_i)+\frac{N}{2} \int g \ast \rho_{f_0} \rho_{f_0}\\
        =&\, \frac{k}{N} k\mathcal{W}_k(X_k;\rho_{f_0})+\frac{N-k}{N} (N-k)\mathcal{W}_{N-k}(X_{[k+1,N]};\rho_{f_0})+\frac{1}{N} \sum_{i=1}^k \sum_{j=k+1}^N g(x_i-x_j).
    \end{align*}
    Here $X_{[k+1,N]}=(x_{k+1}, \cdots, x_N)$. The last term is bounded from below by $-Ck$ since the potential $g \geq -C$. Therefore, we have
    \begin{align*}
        \int_{\Omega^k} |g_{N,k}^0|^q e^{\lambda(0)e_k} \dd Z_k \leq C_{\beta, q}^k \|f_0\|_{L^\infty}^{(q-1)k} &\, \int_{\Omega^k} e^{-q\beta \frac{k}{N} k \mathcal{W}_k(X_k;\rho_{f_0})} f_0^{\otimes k} e^{\Lambda^{-1} e_k} \dd Z_k \\
        \cdot &\, \int_{\Omega^{N-k}} e^{-q\beta \frac{N-k}{N} (N-k) \mathcal{W}_{N-k} (X_{[k+1, N]};\rho_{f_0})} f_0^{\otimes (N-k)} \dd Z_{N-k}.
    \end{align*}
    The second integral is exactly $\overline{Z}_{N-k, (N-k)q\beta/N}$, which is bounded by $C_{\beta, q}$ by Proposition \ref{prop:partition-function}. Using again $g \geq -C$ and $\rho_{f_0} \in L^1 \cap L^\infty(\D)$, we see that, once $\Lambda^{-1}<\min(\beta/2, \alpha)$, the first integral is bounded by
    \begin{equation*}
        C_{q,\beta}^k \|f_0\|_{L^\infty}^{(q-1)k}
        \begin{cases}
            \displaystyle \int_{\Omega^k} f_0^{\otimes k} e^{\Lambda^{-1} (k+\sum_{i=1}^k |v_i|^2)} \dd Z_k, \quad &\, \D=\T^d\\
            \displaystyle \int_{\Omega^k} f_0^{\otimes k} e^{\Lambda^{-1} (k+\sum_{i=1}^k |z_i|^2)} \dd Z_k, \quad &\, \D=\R^d
        \end{cases}
        \leq C^k
    \end{equation*}
    for some $C$ depending only on $q, \beta, M_\alpha, \|f_0\|_{L^\infty}, \|\rho_{f_0}\|_{L^\infty}$, since the non-exponentially integrable weight is absorbed by the Gibbs weight. This proves the first condition in \eqref{eq:bjs-condition} and \eqref{eq:bjs-condition-whole-space}. The second condition follows from taking $k=N$ in the first condition and applying \cite[Lemma 9]{bresch2025new}.

    Therefore, based on Proposition \ref{prop:bjs-estimate} and Proposition \ref{prop:bjs-estimate-whole-space}, we get the desired compactness bound on the marginals $g_{N,k}$. The rest of the proof follows identically to that in \cite{bresch2025new} by the standard compactness-uniqueness argument. This, combined with Theorem \ref{thm:one-initial-data}, finishes the proof.
\end{proof}

\subsection*{Acknowledgements}

The authors would like to thank Hao Liang and Xianliang Zhao for helpful discussions on modulated Gibbs measures for classical particle systems. This work was partially supported by the National Key R\&D Program of China (Project No.~2024YFA1015500) and the National Natural Science Foundation of China (NSFC; Grant Nos.~12595282 and 12171009).

\bibliography{ref}

\begin{thebibliography}{10}

\bibitem{bresch2024duality}
D.~Bresch, M.~Duerinckx, and P.-E. Jabin.
\newblock A duality method for mean-field limits with singular interactions.
\newblock {\em arXiv preprint arXiv:2402.04695}, 2024.

\bibitem{bresch2025new}
D.~Bresch, P.-E. Jabin, and J.~Soler.
\newblock A new approach to the mean-field limit of {V}lasov--{F}okker--{P}lanck equations.
\newblock {\em Analysis \& PDE}, 18(4):1037--1064, 2025.

\bibitem{bresch2019modulated}
D.~Bresch, P.-E. Jabin, and Z.~Wang.
\newblock Modulated free energy and mean field limit.
\newblock {\em S{\'e}minaire Laurent Schwartz-EDP et applications}, pages 1--22, 2019.

\bibitem{bresch2019mean}
D.~Bresch, P.-E. Jabin, and Z.~Wang.
\newblock On mean-field limits and quantitative estimates with a large class of singular kernels: Application to the {P}atlak--{K}eller--{S}egel model.
\newblock {\em Comptes Rendus Mathematique}, 357(9):708--720, 2019.

\bibitem{bresch2023mean}
D.~Bresch, P.-E. Jabin, and Z.~Wang.
\newblock Mean field limit and quantitative estimates with singular attractive kernels.
\newblock {\em Duke Mathematical Journal}, 172(13):2591--2641, 2023.

\bibitem{carrillo2019mean}
J.~A. Carrillo, Y.-P. Choi, M.~Hauray, and S.~Salem.
\newblock Mean-field limit for collective behavior models with sharp sensitivity regions.
\newblock {\em Journal of the European Mathematical Society}, 21(1):121--161, 2019.

\bibitem{carrillo2019propagation}
J.~A. Carrillo, Y.-P. Choi, and S.~Salem.
\newblock Propagation of chaos for the {V}lasov--{P}oisson--{F}okker--{P}lanck equation with a polynomial cut-off.
\newblock {\em Communications in Contemporary Mathematics}, 21(04):1850039, 2019.

\bibitem{chen2026mean}
L.~Chen, J.~Jung, P.~Pickl, and Z.~Wang.
\newblock On the mean-field limit of {V}lasov--{P}oisson--{F}okker--{P}lanck equations.
\newblock {\em Journal of Mathematical Physics}, 67(6), 2026.

\bibitem{delgadino2025sharp}
M.~G. Delgadino and R.~S. Gvalani.
\newblock Sharp mean-field estimates for the repulsive log gas in any dimension.
\newblock {\em arXiv preprint arXiv:2506.22083}, 2025.

\bibitem{delgadino2026sharp}
M.~G. Delgadino, R.~S. Gvalani, and M.~Rosenzweig.
\newblock Sharp mean-field estimates for diffusive log/{R}iesz gases in the {H}ilbert--{S}chmidt regime.
\newblock {\em arXiv preprint arXiv:2609.02743}, 2026.

\bibitem{dobrushin1979vlasov}
R.~L. Dobrushin.
\newblock Vlasov equations.
\newblock {\em Functional Analysis and Its Applications}, 13(2):115--123, 1979.

\bibitem{duerinckx2016mean}
M.~Duerinckx.
\newblock Mean-field limits for some {R}iesz interaction gradient flows.
\newblock {\em SIAM Journal on Mathematical Analysis}, 48(3):2269--2300, 2016.

\bibitem{duerinckx2026derivation}
M.~Duerinckx and P.-E. Jabin.
\newblock Derivation of 2{D} {V}lasov--{P}oisson for classical particles with {C}oulomb interactions.
\newblock {\em arXiv preprint arXiv:2608.04104}, 2026.

\bibitem{duerinckx2026singular}
M.~Duerinckx and P.-E. Jabin.
\newblock Singular mean-field limits for fluctuations around equilibrium.
\newblock {\em arXiv preprint arXiv:2605.28979}, 2026.

\bibitem{dupuis1997weak}
P.~Dupuis and R.~S. Ellis.
\newblock {\em A Weak Convergence Approach to the Theory of Large Deviations}.
\newblock Wiley Series in Probability and Statistics. John Wiley \& Sons, New York, 1997.

\bibitem{golse2016dynamics}
F.~Golse.
\newblock On the dynamics of large particle systems in the mean field limit.
\newblock {\em Macroscopic and large scale phenomena: coarse graining, mean field limits and ergodicity}, pages 1--144, 2016.

\bibitem{hauray2007n}
M.~Hauray and P.-E. Jabin.
\newblock ${N}$-particles approximation of the {V}lasov equations with singular potential.
\newblock {\em Archive for Rational Mechanics and Analysis}, 183(3):489--524, 2007.

\bibitem{hauray2015particle}
M.~Hauray and P.-E. Jabin.
\newblock Particle approximation of {V}lasov equations with singular forces: Propagation of chaos.
\newblock {\em Annales scientifiques de l'{\'E}cole Normale Sup{\'e}rieure}, 48(4):891--940, 2015.

\bibitem{hauray2014kac}
M.~Hauray and S.~Mischler.
\newblock On {K}ac's chaos and related problems.
\newblock {\em Journal of Functional Analysis}, 266(10):6055--6157, 2014.

\bibitem{hauray2019propagation}
M.~Hauray and S.~Salem.
\newblock Propagation of chaos for the {V}lasov--{P}oisson--{F}okker--{P}lanck system in 1{D}.
\newblock {\em Kinetic and Related Models}, 12(2), 2019.

\bibitem{huang2020mean}
H.~Huang, J.-G. Liu, and P.~Pickl.
\newblock On the mean-field limit for the {V}lasov--{P}oisson--{F}okker--{P}lanck system.
\newblock {\em Journal of Statistical Physics}, 181(5):1915--1965, 2020.

\bibitem{jabin2014review}
P.-E. Jabin.
\newblock A review of the mean field limits for {V}lasov equations.
\newblock {\em Kinetic and Related Models}, 7(4):661, 2014.

\bibitem{jabin2016mean}
P.-E. Jabin and Z.~Wang.
\newblock Mean field limit and propagation of chaos for {V}lasov systems with bounded forces.
\newblock {\em Journal of Functional Analysis}, 271(12):3588--3627, 2016.

\bibitem{jabin2017mean}
P.-E. Jabin and Z.~Wang.
\newblock Mean field limit for stochastic particle systems.
\newblock {\em Active Particles, Volume 1: Advances in Theory, Models, and Applications}, pages 379--402, 2017.

\bibitem{jabin2018quantitative}
P.-E. Jabin and Z.~Wang.
\newblock Quantitative estimates of propagation of chaos for stochastic systems with ${W}^{-1,\infty}$ kernels.
\newblock {\em Inventiones mathematicae}, 214:523--591, 2018.

\bibitem{kac1956foundations}
M.~Kac.
\newblock Foundations of kinetic theory.
\newblock In {\em Proceedings of the Third Berkeley Symposium on Mathematical Statistics and Probability, Volume 3: Contributions to Astronomy and Physics}, volume~3, pages 171--198. University of California Press, 1956.

\bibitem{khoury2026quantitative}
N.~Khoury and P.-E. Jabin.
\newblock Quantitative estimates for mean-field limits and correlation functions through a duality framework.
\newblock {\em arXiv preprint arXiv:2605.02058}, 2026.

\bibitem{lacker2023hierarchies}
D.~Lacker.
\newblock Hierarchies, entropy, and quantitative propagation of chaos for mean field diffusions.
\newblock {\em Probability and Mathematical Physics}, 4(2):377--432, 2023.

\bibitem{lazarovici2017mean}
D.~Lazarovici and P.~Pickl.
\newblock A mean field limit for the {V}lasov--{P}oisson system.
\newblock {\em Archive for Rational Mechanics and Analysis}, 225(3):1201--1231, 2017.

\bibitem{mckean1967propagation}
H.~P. McKean.
\newblock Propagation of chaos for a class of non-linear parabolic equations.
\newblock {\em Stochastic Differential Equations (Lecture Series in Differential Equations, Session 7, Catholic Univ., 1967)}, pages 41--57, 1967.

\bibitem{rosenzweig2025modulated}
M.~Rosenzweig and S.~Serfaty.
\newblock Modulated logarithmic {S}obolev inequalities and generation of chaos.
\newblock {\em Annales de la Facult{\'e} des sciences de Toulouse: Math{\'e}matiques}, 34(1):107--134, 2025.

\bibitem{serfaty2020mean}
S.~Serfaty.
\newblock Mean field limit for {C}oulomb-type flows.
\newblock {\em Duke Mathematical Journal}, 169(15):2887--2935, 2020.

\bibitem{sznitman1991topics}
A.-S. Sznitman.
\newblock Topics in propagation of chaos.
\newblock {\em Ecole d'Et{\'e} de Probabilit{\'e}s de Saint-Flour XIX—1989}, pages 165--251, 1991.

\bibitem{wang2026uniform}
Z.~Wang and X.~Zhao.
\newblock Uniform partition-function estimates for {C}oulomb modulated energy at all positive temperatures.
\newblock {\em arXiv preprint arXiv:2609.00867}, 2026.

\bibitem{yau1991relative}
H.-T. Yau.
\newblock Relative entropy and hydrodynamics of {G}inzburg--{L}andau models.
\newblock {\em Letters in Mathematical Physics}, 22(1):63--80, 1991.

\end{thebibliography}
\bibliographystyle{abbrv}

\end{document}